\documentclass[a4paper]{article}
\usepackage{graphicx}
\usepackage{amsmath}
\usepackage{amsfonts}
\usepackage{amsthm}
\usepackage{amssymb}
\usepackage{subfigure}
\usepackage{multirow}
\usepackage{rotating}
\usepackage{fancybox}
\usepackage{boxedminipage}
\usepackage{version}
\usepackage[usenames]{color}
\usepackage[english,greek]{babel}
\usepackage[iso-8859-7]{inputenc}
\newcommand{\ds}{\displaystyle}
\newcommand{\en}{\selectlanguage{english}}

\newcommand{\NN}{\mathbb{N}}

\newcommand{\ZZ}{\mathbb{Z}}

\newcommand{\FF}{\mathbb{F}}

\newcommand{\be}{\begin{equation}}
\newcommand{\ee}{\end{equation}}

\newtheorem{thm}{Theorem}

\theoremstyle{definition}
\newtheorem{rmk}{Remark}
\newtheorem{dfn}{Definition}
\newtheorem{con}{Conjecture}

\begin{document} \en
\title{Three experimental pearls in Costas arrays}
\author{Konstantinos Drakakis\footnote{The author holds a Diploma in Electrical and Computer Engineering from NTUA, Athens, Greece, and a Ph.D. in Applied and Computational Mathematics from Princeton University, NJ, USA. He is also affiliated with Electronic \& Electrical Engineering, University College Dublin, Ireland.\newline  \textbf{Address:} Electronic \& Electrical Engineering, University College Dublin, Belfield, Dublin 4, Ireland\newline \textbf{Email:} Konstantinos.Drakakis@ucd.ie}
\\ Claude Shannon Institute\footnote{www.shannoninstitute.ie}\\Ireland}
\maketitle

\abstract{The results of 3 experiments in Costas arrays are presented, for which theoretical explanation is still not available: the number of dots on the main diagonal of exponential Welch arrays, the parity populations of Golomb arrays generated in fields of characteristic 2, and the maximal cross-correlation between pairs of Welch or Golomb arrays generated in fields of size equal to a Sophie Germain prime.}

\section{Introduction}

Costas arrays appeared for the first time in 1965 in the context of SONAR detection (\cite{C1}, and later \cite{C2} as a journal publication), when J. P. Costas, disappointed by the poor performance of SONARs, used them to describe a novel frequency hopping pattern for SONARs with optimal auto-correlation properties. At that stage their study was entirely empirical and application-oriented. In 1984, however, after the publication by S. Golomb \cite{G} of the 2 main construction methods for Costas arrays (the Welch and the Golomb algorithm) based on finite fields, still the only ones available today, they officially acquired their present name and they became an object of mathematical interest and study.

Soon it became clear that the mathematical problems related to Costas arrays presented a challenge for our present methodology in Discrete Mathematics (Combinatorics, Algebra, Number Theory etc.), and suggested that novel techniques are desperately needed, perhaps currently lying beyond the frontiers of our knowledge. Indeed, we have so far been unable to settle even the most fundamental question in the field: the issue of the existence of Costas arrays for all orders.

These insurmountable difficulties the researchers were faced with triggered inevitably an intense activity in computer exploration of Costas arrays (for example \cite{BREMW, BCGRMP, DGRT, SVM}), the rationale being that it is easier to prove something on which strong evidence has been gathered, rather than starting completely from scratch. Such evidence led to the formulation of conjectures, some of which subsequently were, at least partially, proved. These successes, however small, helped consolidate the position of the experimental method as an indispensable tool for the study of Costas arrays.

Not all computer experiments have led to conjectures, however, let alone successfully proved conjectures: several experiments, perhaps the most interesting ones, yielded results that still defy any attempt for explanation. In this work we collect our findings in 3 numerical experiments we performed on Costas arrays, whose results appear very interesting, but entirely inexplicable at present, and we present them to the broader scientific community, hoping to accelerate progress towards their solution. They are:
\begin{itemize}
  \item The number of dots on the main diagonal of exponential Welch arrays;
  \item The parity populations of Golomb arrays generated in fields of characteristic 2;
  \item and the maximal cross-correlation between pairs of Welch or Golomb arrays generated in fields of size equal to a Sophie Germain prime.
\end{itemize}

The reason for choosing these particular 3 experiments is that, having spent lots of time studying them, we can confidently say that a lot more is to be gained than mere deeper understanding of Costas arrays through their successful explanation: in our opinion, such an explanation relies on completely novel, as yet unexplored areas of finite fields, and traditional algebraic and number theoretic methods are totally incapable of making any progress. In other words, these problems, although originating in the relatively unknown field of Costas arrays, reveal new directions in Algebra and Number Theory, and are, consequently, of paramount pure mathematical interest.

\section{Basics}

In this section we give precise definitions for all the terms used in the introduction, as well as for everything else needed in the paper.

\subsection{Definition of the Costas property}

Simply put, a Costas array is a square arrangement of dots and blanks, such that there is exactly one dot per row and column, and such that all vectors between dots are distinct.

\begin{dfn}
Let $f:[n]\rightarrow [n]$, where $[n]=\{1,\ldots,n\}$, $n\in\NN$, be a bijection; then $f$ has the \emph{Costas property} iff the collection of vectors $\{(i-j,f(i)-f(j)): 1\leq j<i\leq n\}$, called \emph{the distance vectors}, are all distinct, in which case $f$ is called a \emph{Costas permutation}. The corresponding \emph{Costas array} $A_f$ is the square array $n\times n$ where the elements at $(f(i),i),\ i\in[n]$  are equal to 1 (dots), while the remaining elements are equal to 0 (blanks):
\[A_f=[a_{ij}]=\begin{cases} 1\text{ if } i=f(j)\\ 0\text{ otherwise}\end{cases},\ j\in[n]\]
\end{dfn}

\begin{rmk}\
The operations of horizontal flip, vertical flip, and transposition on a Costas array result to a Costas array as well: hence, out of a Costas array 8 can be created, or 4 if the particular Costas array is symmetric.
\end{rmk}

\subsection{Construction algorithms}

There are 2 known algorithms for the construction of Costas arrays. We state them below omitting the proofs (which can be found in \cite{D,G} in full detail):

\begin{dfn}[Exponential Welch construction $W_1(p,g,c)$] \label{w}
Let $p$ be a prime, $g$  a primitive root of the finite field $\FF(p)$, and $c\in[p-1]-1$; the \emph{exponential Welch permutation} corresponding to $g$ and $c$ is defined by $\ds f(i)=g^{i-1+c}\mod p,\ i\in[p-1]$.
\end{dfn}

\begin{rmk}
Given a $W_1$ permutation, it is well known that its horizontal and vertical flips also correspond to $W_1$ permutations; its transpose, however, does not: it is what we define as a \emph{logarithmic Welch permutation}. The distinction is well defined as, for $p>5$, there are no symmetric $W_2$ arrays. We will no further consider logarithmic Welch permutations in this work, so ``Welch'' will henceforth be synonymous to ``exponential Welch''.
\end{rmk}

\begin{dfn}[Golomb construction $G_2(p,m,a,b)$]
Let $q=p^m$, where $p$ prime and $m\in\NN^*$, and let $a$, $b$ be primitive roots of the finite field $\FF(q)$; the Golomb permutation corresponding to $a$ and $b$ is defined through the equation $a^i+b^{f(i)}=1,\ i\in[q-2]$.
\end{dfn}

\begin{rmk}
The horizontal and vertical flips of a $G_2$ permutation are themselves $G_2$ permutations, just like in the Welch case; this time, however, the same holds true for transpositions as well.
\end{rmk}

\begin{rmk}
The indices in $W_1$ and $G_2$ have the significance that the methods produce permutations of orders 1 and 2 smaller than the size of the finite field they use, respectively. It is well known that both methods can be extended to yield a wide range of sub-methods; in this paper, however, we will focus exclusively on the 2 aforementioned main construction methods.
\end{rmk}

\subsection{Cross-correlation}

We now give a precise definition of the cross-correlation between 2 Costas arrays:

\begin{dfn}
Let $f,g:[n]\rightarrow [n]$ where $n\in\NN^*$, and let $u,v\in\ZZ$; the cross-correlation between $f$ and $g$ at $(u,v)$ is defined as
\be \label{ccdef} \Psi_{f,g}(u,v)=|\{(f(i)+v,i+u):i\in[n]\}\cap \{(g(i),i):i\in[n]\}|\ee
\end{dfn}

Informally, we can think of the cross-correlation in the following way: first, we place the 2 Costas arrays on top of each other, and then we translate the first by $v$ units vertically and $u$ horizontally; the number of pairs of overlapping dots in this position is the value of the cross-correlation at $(u,v)$.

\begin{rmk}
If either one of the 2 ``+'' signs in (\ref{ccdef}) is interpreted as modulo addition, the cross-correlation becomes periodic in the corresponding direction. More precisely, in the treatment of the Welch case that is about to follow, it is natural to interpret $i+u$ as a modulo addition, so that the cross-correlation between $W_1$ arrays is periodic in the horizontal direction, while in the dovetailing treatment of the Golomb case it is natural to interpret both $i+u$ and $f(i)+v$ as modulo additions, so that the cross-correlation between $G_2$ arrays is periodic in both directions.
\end{rmk}

\subsection{Parity populations}

\begin{dfn}
Let $f:[n]\rightarrow [n]$, $n\in\NN^*$, be a function; set:
\begin{itemize}
  \item $ee(f)=|\{i\in[n]: i\mod 2=f(i)\mod 2=0\}|$ to be the \emph{even-even population};
  \item $oo(f)=|\{i\in[n]: i\mod 2=f(i)\mod 2=1\}|$ to be the \emph{odd-odd population};
  \item $eo(f)=|\{i\in[n]: i\mod 2=1,\ f(i)\mod 2=0\}|$ to be the \emph{even-odd population};
  \item $oe(f)=|\{i\in[n]: i\mod 2=0,\ f(i)\mod 2=1\}|$ to be the \emph{odd-even population};
\end{itemize}
\end{dfn}

If $f$ is a permutation, the parity populations are closely connected:

\begin{thm}
Let $f:[n]\rightarrow [n]$, $n\in\NN^*$, be a permutation; then
\begin{itemize}
  \item $ee(f)+oo(f)+eo(f)+oe(f)=n$;
  \item $oe(f)=eo(f)$;
  \item $oo(f)-ee(f)=n\mod 2$.
\end{itemize}
\end{thm}

\begin{proof}
This is actually a very simple, almost obvious result. Clearly, $ee+eo=ee+oe$, as both sums equal the number of even integers in $[n]$; hence, $eo=oe$. Further, $oo+oe$ is the number of odd integers in $[n]$, whence:
\[oo+oe-(ee+eo)=oo-ee=\begin{cases} 1\text{ if }n\mod2\equiv1\\ 0\text{ if }n\mod2\equiv0=n\mod 2\end{cases}\]
\end{proof}

There is then only one degree of freedom: if one of the populations is given, all 4 can be determined.

\subsection{Sophie Germain primes}

A very special family of primes will play an important role in Section \ref{s3}:

\begin{dfn}
A prime $p$ is a \emph{Sophie Germain prime} \cite{S} iff $p=2q+1$, where $q$ prime.
\end{dfn}

It is not known whether this family contains infinitely many primes, although it is conjectured to do so.

\section{The number of dots on the main diagonal of exponential Welch arrays}\label{s1}

In accordance with Definition \ref{w}, given a prime $p$, we are interested in the number of solutions of
\be\label{bwed} i\equiv g^{i-1+c}\mod p\ee
with respect to $i$, where $g$ is a primitive root of the field $\FF(p)$ and $c\in[p-1]-1$ is a constant.

Equation (\ref{bwed}) strikes one immediately as ``unalgebraic'': the $i$ on the RHS is simply an index, and in particular an integer in $[p-1]-1$, based on Fermat's Little Theorem; the $i$ on the LHS, however, is an element of $\FF(p)$, and elements of $\FF(p)$ just happen to be representable by integers because $\FF(p)$ is a field of prime size and not an extension field (whose elements are routinely represented as polynomials). In other words, Algebra traditionally considers the 2 instances of $i$ in (\ref{bwed}) as different, non-comparable objects, and these 2 object types happen to coincide in finite fields of prime size; the solution of this equation then needs to exploit properties of these fields not present in extension fields, where this equation is impossible to formulate in the first place, and this probably means that we need to consider $\FF(p)$ as something more complex than a field.

The bottom line is that we are left with a transcendental equation over a finite field. Such equations have almost not been studied at all, as opposed to polynomial equations, on which the literature is abundant. The only instance of a relevant problem studied in the literature (that we have been able to trace) has been one proposed by Demetrios Brizolis: is it true that $\ds \forall i\in[p-1]\ \exists g\in[p-1]:\ i\equiv g^i\mod p$? This was answered in the affirmative by W. P. Zhang \cite{Z} for sufficiently large primes, and later C. Pomerance and M. Campbell ``\emph{made the value of ``sufficiently large'' small enough that they were able to use a direct search to affirmatively answer Brizolis' original question}'' (\cite{HM} and references therein). Observe, though, that this is quite a different problem than the one we are interested in.

Let $\ds S(p,g,c)=\left|\left\{i\in[p-1]:i\equiv g^{i-1+c}\right\}\right|$, namely the number of solutions of (\ref{bwed}) for a given constant $c$ and a primitive root $g\in\FF(p)$, $p$ prime. Table \ref{t1} shows $\ds \max_{(g,c)} S(p,g,c)$ for all $p<5000$: the data do not seem to follow a recognizable pattern, but they roughly seem to behave ``logarithmically''\footnote{A graph of these results for $p<1000$ was presented in \cite{DGO}}. Indeed, $1+[\ln(p)]$, where $[\cdot]$ is the rounding function, seems to fit the data very well: 402 out of 669 entries (60.1\%) are captured exactly, while 652 entries (97.5\%) are captured within an error margin of $\pm 1$.

Finally, here is an interesting additional side observation we made during our experiments: it is a well known result in Combinatorics (``The problem of the misadressed letters'') that the ratio of permutations of order $n$ without fixed points over the total $n!$ permutations approaches $e^{-1}=0.3678794\ldots$ as $n\rightarrow \infty$. What can be said about the ratio of the population of $W_1$ permutations with no fixed points at all generated in $\FF(p)$  over the totality of $(p-1)\phi(p-1)$ $W_1$ arrays? It is plotted in Fig. \ref{fg} and seems to approach $e^{-1}$ as well, although the data shows still some fluctuation in the given range of $p$.

\begin{figure}
\centering
\includegraphics[height=250pt]{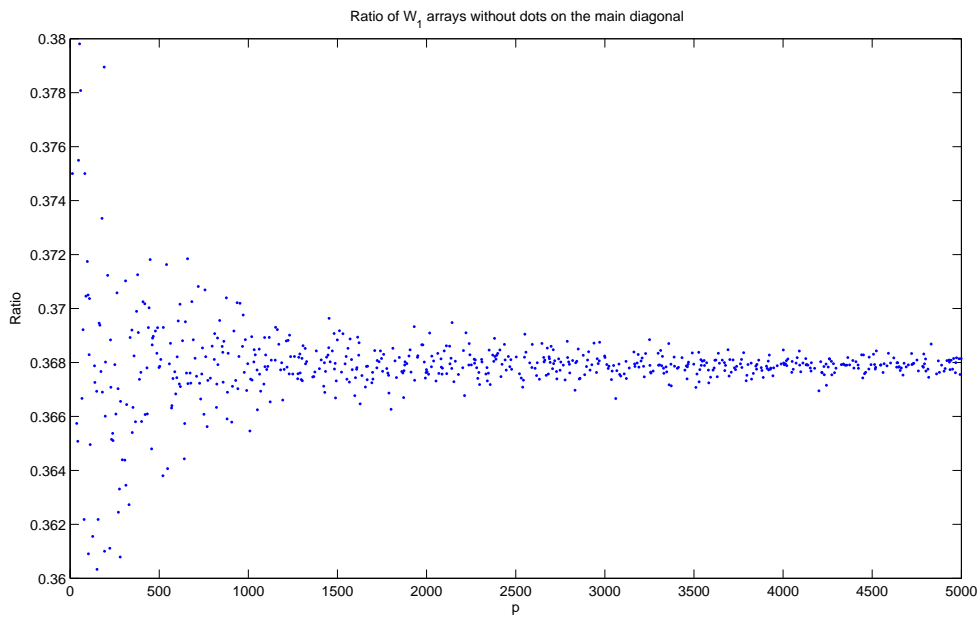}
\caption{\label{fg} Plot of the ratio of $W_1$ arrays with no dots on the main diagonal over the total number of $W_1$ arrays generated in $\FF(p)$ as a function of $p$}
\end{figure}

\begin{table}

{\scriptsize

\begin{tabular}{|r|r||r|r||r|r||r|r||r|r||r|r||r|r||r|r||r|r||r|r|}
\hline
$p$ & $\#$ & $p$ & $\#$ & $p$ & $\#$ & $p$ & $\#$ & $p$ & $\#$ & $p$ & $\#$ & $p$ & $\#$ & $p$ & $\#$ & $p$ & $\#$ & $p$ & $\#$\\\hline\hline
2&1&337&6&761&7&1231&7&1723&9&2267&9&2767&8&3331&8&3877&9&4447&9\\\hline
3&2&347&6&769&8&1237&7&1733&8&2269&10&2777&9&3343&8&3881&9&4451&10\\\hline
5&2&349&7&773&9&1249&8&1741&8&2273&10&2789&8&3347&9&3889&10&4457&9\\\hline
7&3&353&8&787&8&1259&8&1747&9&2281&8&2791&9&3359&9&3907&11&4463&9\\\hline
11&4&359&7&797&7&1277&8&1753&9&2287&9&2797&8&3361&8&3911&9&4481&10\\\hline
13&4&367&8&809&8&1279&8&1759&9&2293&9&2801&9&3371&9&3917&9&4483&9\\\hline
17&3&373&7&811&8&1283&8&1777&9&2297&9&2803&9&3373&11&3919&9&4493&9\\\hline
19&5&379&7&821&9&1289&9&1783&8&2309&9&2819&10&3389&9&3923&10&4507&9\\\hline
23&5&383&7&823&7&1291&7&1787&9&2311&9&2833&9&3391&10&3929&9&4513&10\\\hline
29&4&389&7&827&7&1297&9&1789&8&2333&9&2837&9&3407&10&3931&8&4517&9\\\hline
31&4&397&7&829&9&1301&8&1801&9&2339&9&2843&9&3413&9&3943&9&4519&9\\\hline
37&4&401&7&839&8&1303&7&1811&9&2341&9&2851&8&3433&9&3947&10&4523&9\\\hline
41&5&409&7&853&8&1307&8&1823&8&2347&9&2857&8&3449&9&3967&9&4547&9\\\hline
43&4&419&7&857&8&1319&9&1831&8&2351&10&2861&8&3457&10&3989&10&4549&9\\\hline
47&5&421&6&859&8&1321&8&1847&8&2357&8&2879&8&3461&9&4001&10&4561&9\\\hline
53&5&431&7&863&8&1327&8&1861&9&2371&8&2887&9&3463&9&4003&9&4567&10\\\hline
59&5&433&7&877&7&1361&8&1867&9&2377&9&2897&9&3467&9&4007&10&4583&9\\\hline
61&5&439&8&881&8&1367&8&1871&8&2381&9&2903&8&3469&8&4013&10&4591&9\\\hline
67&5&443&7&883&8&1373&9&1873&9&2383&9&2909&9&3491&10&4019&9&4597&9\\\hline
71&5&449&7&887&8&1381&9&1877&9&2389&8&2917&8&3499&9&4021&9&4603&9\\\hline
73&5&457&7&907&7&1399&7&1879&8&2393&8&2927&9&3511&9&4027&9&4621&9\\\hline
79&5&461&8&911&7&1409&8&1889&8&2399&9&2939&9&3517&10&4049&9&4637&9\\\hline
83&6&463&7&919&9&1423&7&1901&10&2411&9&2953&9&3527&9&4051&9&4639&9\\\hline
89&5&467&7&929&8&1427&9&1907&8&2417&10&2957&9&3529&9&4057&9&4643&10\\\hline
97&6&479&8&937&8&1429&8&1913&8&2423&11&2963&9&3533&9&4073&10&4649&10\\\hline
101&6&487&8&941&8&1433&8&1931&9&2437&10&2969&9&3539&9&4079&9&4651&9\\\hline
103&6&491&7&947&8&1439&8&1933&8&2441&8&2971&9&3541&9&4091&10&4657&9\\\hline
107&6&499&7&953&9&1447&9&1949&9&2447&9&2999&9&3547&8&4093&10&4663&9\\\hline
109&6&503&7&967&8&1451&9&1951&8&2459&8&3001&9&3557&9&4099&9&4673&9\\\hline
113&5&509&7&971&8&1453&9&1973&10&2467&8&3011&9&3559&9&4111&10&4679&9\\\hline
127&5&521&7&977&8&1459&8&1979&10&2473&9&3019&9&3571&9&4127&10&4691&9\\\hline
131&6&523&8&983&9&1471&8&1987&8&2477&9&3023&11&3581&9&4129&10&4703&9\\\hline
137&6&541&7&991&9&1481&8&1993&9&2503&9&3037&8&3583&9&4133&9&4721&10\\\hline
139&6&547&7&997&10&1483&8&1997&9&2521&9&3041&9&3593&9&4139&9&4723&10\\\hline
149&6&557&7&1009&7&1487&8&1999&8&2531&9&3049&9&3607&10&4153&9&4729&9\\\hline
151&5&563&8&1013&8&1489&8&2003&8&2539&9&3061&9&3613&9&4157&9&4733&9\\\hline
157&5&569&8&1019&8&1493&9&2011&9&2543&9&3067&9&3617&9&4159&10&4751&10\\\hline
163&6&571&7&1021&7&1499&8&2017&9&2549&9&3079&9&3623&11&4177&9&4759&9\\\hline
167&7&577&7&1031&8&1511&8&2027&9&2551&8&3083&9&3631&9&4201&9&4783&10\\\hline
173&6&587&8&1033&8&1523&10&2029&9&2557&9&3089&9&3637&9&4211&9&4787&11\\\hline
179&7&593&7&1039&8&1531&8&2039&9&2579&10&3109&9&3643&10&4217&9&4789&10\\\hline
181&5&599&7&1049&8&1543&8&2053&9&2591&9&3119&9&3659&10&4219&9&4793&10\\\hline
191&6&601&7&1051&8&1549&8&2063&10&2593&9&3121&9&3671&9&4229&10&4799&9\\\hline
193&6&607&8&1061&8&1553&8&2069&9&2609&9&3137&9&3673&9&4231&9&4801&9\\\hline
197&8&613&8&1063&8&1559&9&2081&9&2617&10&3163&8&3677&9&4241&9&4813&9\\\hline
199&6&617&8&1069&7&1567&9&2083&9&2621&8&3167&10&3691&9&4243&11&4817&10\\\hline
211&6&619&8&1087&7&1571&9&2087&9&2633&10&3169&9&3697&9&4253&11&4831&10\\\hline
223&7&631&7&1091&8&1579&8&2089&8&2647&9&3181&9&3701&8&4259&11&4861&10\\\hline
227&6&641&7&1093&7&1583&9&2099&9&2657&9&3187&9&3709&9&4261&9&4871&9\\\hline
229&5&643&8&1097&8&1597&8&2111&8&2659&9&3191&10&3719&9&4271&9&4877&10\\\hline
233&6&647&9&1103&8&1601&8&2113&9&2663&9&3203&10&3727&8&4273&9&4889&11\\\hline
239&8&653&7&1109&8&1607&8&2129&9&2671&8&3209&8&3733&9&4283&10&4903&9\\\hline
241&7&659&7&1117&7&1609&9&2131&9&2677&9&3217&10&3739&9&4289&9&4909&8\\\hline
251&6&661&7&1123&8&1613&8&2137&8&2683&9&3221&10&3761&9&4297&9&4919&10\\\hline
257&7&673&7&1129&8&1619&9&2141&9&2687&9&3229&9&3767&9&4327&9&4931&8\\\hline
263&6&677&9&1151&7&1621&9&2143&8&2689&9&3251&9&3769&10&4337&9&4933&9\\\hline
269&7&683&7&1153&8&1627&8&2153&9&2693&9&3253&9&3779&9&4339&8&4937&9\\\hline
271&6&691&7&1163&8&1637&9&2161&8&2699&9&3257&10&3793&8&4349&9&4943&9\\\hline
277&6&701&7&1171&9&1657&8&2179&8&2707&8&3259&9&3797&9&4357&9&4951&9\\\hline
281&7&709&8&1181&8&1663&8&2203&9&2711&9&3271&9&3803&9&4363&9&4957&9\\\hline
283&6&719&7&1187&9&1667&8&2207&9&2713&9&3299&8&3821&10&4373&10&4967&9\\\hline
293&7&727&8&1193&8&1669&9&2213&10&2719&8&3301&9&3823&9&4391&11&4969&9\\\hline
307&7&733&8&1201&7&1693&8&2221&9&2729&9&3307&9&3833&10&4397&9&4973&10\\\hline
311&6&739&8&1213&8&1697&8&2237&9&2731&9&3313&9&3847&10&4409&9&4987&9\\\hline
313&7&743&7&1217&8&1699&8&2239&8&2741&9&3319&8&3851&10&4421&9&4993&9\\\hline
317&6&751&7&1223&8&1709&8&2243&10&2749&9&3323&8&3853&9&4423&9&4999&9\\\hline
331&6&757&8&1229&8&1721&8&2251&8&2753&8&3329&10&3863&10&4441&10&&\\\hline
\end{tabular}
}
\caption{\label{t1} The maximum number of solutions of the equation $\ds i\equiv g^{i-1+c}\mod p$ over all possible values of $c$ and primitive roots $g\in\FF(p)$, $p<5000$.}
\end{table}

\section{The parity populations of Golomb arrays generated in fields of characteristic 2}

The parity populations for both $W_1$ and $G_2$ arrays generated in fields of odd characteristic have already been completely described \cite{DGR}:

\begin{thm} Let a permutation be generated by $G_2(p,m,a,b)$, $p>2$, $q=p^m$. Then:
\begin{itemize}
	\item If $\ds q\equiv 1\mod 4\Rightarrow ee=\frac{q-5}{4},\ eo=oe=oo=\frac{q-1}{4}$;
	\item If $\ds q\equiv 3\mod 4\Rightarrow oo=\frac{q+1}{4},\ eo=oe=ee=\frac{q-3}{4}$.
\end{itemize}
\end{thm}

\begin{thm} Let  permutation be generated by $W_1(p,g,0)$. Then:
\begin{itemize}
\item If $p\equiv 1\mod4\Rightarrow ee=oo=eo=oe$;
\item If $p\equiv 3\mod8$, then $eo-ee=-3h(-p)$;
\item If $p\equiv 7\mod8$, then $eo-ee=h(-p)$,
\end{itemize}
where $h(-p)$ is the \emph{Class Number for discriminant $-p$}.
\end{thm}

For $p>3$, $\ds h(-p)=-\frac{1}{p}\sum_{i=1}^{p-1}\left(\frac{i}{p}\right)i$, where $\left(\frac{\cdot}{\cdot}\right)$ denotes the Legendre symbol \cite{S}.

Although the proofs (omitted here, but see \cite{DGR} for details) are not necessarily easy (in particular the parity populations of Welch arrays involve the quite advanced concept of the Class Number \cite{BS}), the statements certainly are: the parity populations of $G_2$ arrays generated in $\ds \FF(p^m)$, $p>2$, are independent of the primitive roots $a$ and $b$ used. The same holds essentially true for $W_1$ arrays, except that changing the value of $c$ by 1 causes $ee$ and $eo$ to swap values; as $W_1$ arrays are of even order, horizontal or vertical flips have the same effect, changing the parity of the corresponding coordinate of the dots.

This uniformity holds no longer true for $G_2$ arrays generated in fields of characteristic 2: here, the parity populations can take many different values, which appear to follow no readily recognizable pattern. As these arrays have even order, however, the same phenomenon that we observed in $W_1$ arrays applies here: for each array with parity populations $ee$ and $eo$, there exists another (its horizontal and vertical flip) with these values swapped; hence, there as many arrays with $ee=x$ and $eo=y$ as with $ee=y$ and $eo=x$. The different parity populations observed in $G_2$ arrays generated in the fields of size $2^m,\ m=3,\ldots,11$ are shown in detail in Table \ref{t2}; due to the symmetry we just mentioned, only (the top) half of the array is shown.

Table \ref{t2} shows only the simplest instance of a general phenomenon: consider $k\in\NN^*$ and consider the generalized parity populations modulo $k$. If $k$ happens to be a prime, then the $G_2$ arrays generated in fields of characteristic $k$ exhibit similar behavior. Clearly, Table \ref{t2} corresponds to the first case $k=2$. As we have not experimented extensively with $k>2$, however, we avoid presenting any results at this time.

\begin{table}

\begin{tabular}{ccccc}

\begin{tabular}{c}

\begin{tabular}{|r|r|r|}
\hline
\multicolumn{3}{|c|}{$m=3$}\\\hline
$ee$ & $eo$ & \#\\\hline
1&2&6\\\hline
\end{tabular}
\smallskip
\\
\begin{tabular}{|r|r|r|}
\hline
\multicolumn{3}{|c|}{$m=4$}\\\hline
$ee$ & $eo$ & \#\\\hline
2&5&4\\\hline
3&4&4\\\hline
\end{tabular}
\smallskip
\\
\begin{tabular}{|r|r|r|}
\hline
\multicolumn{3}{|c|}{$m=5$}\\\hline
$ee$ & $eo$ & \#\\\hline
5&10&10\\\hline
6&9&40\\\hline
7&8&40\\\hline
\end{tabular}
\smallskip
\\
\begin{tabular}{|r|r|r|}
\hline
\multicolumn{3}{|c|}{$m=6$}\\\hline
$ee$ & $eo$ & \#\\\hline
12&19&12\\\hline
13&18&22\\\hline
14&17&54\\\hline
15&16&20\\\hline
\end{tabular}
\smallskip
\\
\begin{tabular}{|r|r|r|}
\hline
\multicolumn{3}{|c|}{$m=7$}\\\hline
$ee$ & $eo$ & \#\\\hline
24&39&4\\\hline
25&38&20\\\hline
26&37&44\\\hline
27&36&104\\\hline
28&35&140\\\hline
29&34&206\\\hline
30&33&336\\\hline
31&32&280\\\hline
\end{tabular}
\smallskip
\\
\begin{tabular}{|r|r|r|}
\hline
\multicolumn{3}{|c|}{$m=8$}\\\hline
$ee$ & $eo$ & \#\\\hline
53&74&10\\\hline
54&73&4\\\hline
55&72&12\\\hline
56&71&36\\\hline
57&70&62\\\hline
58&69&106\\\hline
59&68&156\\\hline
60&67&116\\\hline
61&66&166\\\hline
62&65&178\\\hline
63&64&178\\\hline
\end{tabular}

\end{tabular}

&

\begin{tabular}{|r|r|r|}
\hline
\multicolumn{3}{|c|}{$m=9$}\\\hline
$ee$ & $eo$ & \#\\\hline
110&145&8\\\hline
111&144&8\\\hline
112&143&32\\\hline
113&142&26\\\hline
114&141&90\\\hline
115&140&112\\\hline
116&139&156\\\hline
117&138&350\\\hline
118&137&426\\\hline
119&136&496\\\hline
120&135&668\\\hline
121&134&756\\\hline
122&133&872\\\hline
123&132&1020\\\hline
124&131&1232\\\hline
125&130&1296\\\hline
126&129&1436\\\hline
127&128&1384\\\hline
\end{tabular}
&
\begin{tabular}{|r|r|r|}
\hline
\multicolumn{3}{|c|}{$m=10$}\\\hline
$ee$ & $eo$ & \#\\\hline
229&282&2\\\hline
230&281&4\\\hline
231&280&4\\\hline
232&279&16\\\hline
233&278&38\\\hline
234&277&34\\\hline
235&276&60\\\hline
236&275&62\\\hline
237&274&142\\\hline
238&273&164\\\hline
239&272&248\\\hline
240&271&354\\\hline
241&270&326\\\hline
242&269&532\\\hline
243&268&560\\\hline
244&267&792\\\hline
245&266&832\\\hline
246&265&874\\\hline
247&264&972\\\hline
248&263&1130\\\hline
249&262&1276\\\hline
250&261&1282\\\hline
251&260&1524\\\hline
252&259&1620\\\hline
253&258&1654\\\hline
254&257&1718\\\hline
255&256&1780\\\hline
\end{tabular}

&

\begin{tabular}{|l|l|l|}
\hline
\multicolumn{3}{|c|}{$m=11$}\\\hline
$ee$ & $eo$ & \#\\\hline
472&551&4\\\hline
473&550&16\\\hline
475&548&4\\\hline
476&547&4\\\hline
477&546&56\\\hline
478&545&72\\\hline
479&544&120\\\hline
480&543&136\\\hline
481&542&224\\\hline
482&541&348\\\hline
483&540&444\\\hline
484&539&488\\\hline
485&538&782\\\hline
486&537&908\\\hline
487&536&1340\\\hline
488&535&1400\\\hline
489&534&1730\\\hline
490&533&2090\\\hline
491&532&2732\\\hline
492&531&3020\\\hline
493&530&3466\\\hline
494&529&4062\\\hline
495&528&4752\\\hline
496&527&5300\\\hline
497&526&5774\\\hline
498&525&6226\\\hline
499&524&6948\\\hline
500&523&7232\\\hline
501&522&7946\\\hline
502&521&8442\\\hline
503&520&8932\\\hline
504&519&9244\\\hline
505&518&9426\\\hline
506&517&10180\\\hline
507&516&10952\\\hline
508&515&10848\\\hline
509&514&11790\\\hline
510&513&11306\\\hline
511&512&11624\\\hline
\end{tabular}

&
\begin{tabular}{|r|r|}
\hline
$m$ & Length\\\hline
3 & 1\\\hline
4 & 2\\\hline
5 & 3\\\hline
6 & 4\\\hline
7 & 8\\\hline
8 & 11\\\hline
9 & 18\\\hline
10& 27\\\hline
11& 39\\\hline
\end{tabular}

\end{tabular}

\caption{\label{t2} The various different parity populations for $G_2$ arrays generated in $\ds \FF(2^m),\ m=3,\ldots,11$: the third column of each array shows the number of $G_2$ arrays with the given $ee$ and $eo$. The last array contains the number of different parity populations appearing for each given $m$, namely the number of rows of the array corresponding to $m$. Note that the bottom half of the arrays, which is the same as the top half but with the values of $ee$ and $eo$ swapped, is omitted.}

\end{table}

\section{The maximal cross-correlation between pairs of $W_1$ or $G_2$ arrays generated in Germain prime finite fields} \label{s3}

We have at present a fairly good understanding of the maximal cross-correlation between pairs of $W_1$ or $G_2$ arrays generated in a certain field. We have not been able yet to prove these results in their entirety, but we have been able to formulate conjectures that complete the rest of the picture. These conjectures have been based on numerical data that almost leave no doubt about their truth. We present our result below in a series of theorems and conjectures, in order to show clearly what has been proved rigorously so far (for proofs and more details see \cite{DGRT}):

\begin{thm}\label{mainthmw}
Let $f_1$ and $f_2$ be $W_1$ permutations generated in $\FF(p)$, $p$ prime; then, $\ds \max_{(f_1\neq f_2)}\Psi_{f_1,f_2}(u,0)=\frac{p-1}{q}$, where $q$ is the smallest prime dividing $\ds \frac{p-1}{2}$. More precisely, if $\ds \{p_i\}_{i=0}^\infty$ is the sequence of primes, $q=p_I$, $I\in\NN$, iff
\[p\equiv a_0\frac{M_I}{4}\left[\left(\frac{M_I}{4}\right)^{-1}\mod 4\right]+\sum_{i=1}^{I} a_i\frac{M_I}{p_i}\left[\left(\frac{M_I}{p_i}\right)^{-1}\mod p_i\right]\mod M_I\]
where $\ds M_I=4\prod_{i=1}^I p_i$, $a_i\in\{2,\ldots,p_i-1\},\ 0<i<I$, $a_I=1$, and $a_0=3$ if $I>0$.
\end{thm}

\begin{rmk}
In the theorem above, as well as in all subsequent results concerning the cross-correlation of $W_1$ arrays, the condition $f_1\neq f_2$ will be taken to imply the somewhat stronger condition that $f_1$ ad $f_2$ are generated by different primitive roots (as opposed to, say, being generated by the same $g$ but different values of $c$). As the context makes this always clear, we find it unnecessary to use a different symbol and clutter the notation.
\end{rmk}

\begin{thm}\label{gw}
Let $f_1$ and $f_2$ be $W_1$ permutations generated in $\FF(p)$, $p$ prime, and let $f'_1$ and $f'_2$ be $G_2$ permutations generated in the same field by some primitive roots $a$ and $b$, and $a^r$ and $b$, respectively, where $(r,p-1)=1$, $r>1$. Then:
\be \max_{\left(f_1\neq f_2\right)}\Psi_{f_1, f_2}(0,0)=\max_{(f'_1\neq f'_2)}\Psi_{f'_1,f'_2}(0,0)+1\ee
\end{thm}

\begin{con}\label{wcon}
Let $f_1$ and $f_2$ be $W_1$ permutations generated in $\FF(p)$, $p$ prime but not a Germain prime. Then:
\be \max_{(u,v)}\max_{\left(f_1\neq f_2\right)}\Psi_{f_1,f_2}(u,v)=\max_{(f_1\neq f_2)}\Psi_{f_1,f_2}(0,0)\ee
\end{con}

\begin{con}
Let $f_1$ and $f_2$ be $G_2$ permutations in $\FF(p)$, where $p\neq 19$ is prime but not a Germain prime, and similarly let $f'_1$ and $f'_2$ be $G_2$ permutations in the same field with the second primitive root in common. Then:
\be \max_{(u,v)}\max_{(f_1\neq f_2)}\Psi_{f_1,f_2}(u,v)= \max_{(f'_1\neq f'_2)}\Psi_{f'_1,f'_2}(0,0)\ee
\end{con}

Combining all of the above, we can formulate a very strong result:

\begin{con}\label{gcon}
Let $p\neq 19$ be a prime other than a Germain prime, let $f_1$ and $f_2$ be $W_1$ permutations, and let $f'_1$ and $f'_2$ be $G_2$ permutations, all generated in $\FF(p)$. Then:
\be \label{gweq} \max_{(u,v)}\max_{\left(f_1\neq f_2\right)}\Psi_{f_1,f_2}(u,v)=\max_{(u,v)}\max_{(f'_1\neq f'_2)}\Psi_{f'_1,f'_2}(u,v)+1\ee
\end{con}

We will not present here the data that led to the formulation of these conjectures, as they conform entirely with what the conjectures suggest; instead, we wish to present data about what the conjectures do not cover, namely the Sophie Germain primes. We also ignore 19, which seems to be an one-time half-irregular occurrence (regular for $W_1$ arrays but irregular for $G_2$ arrays).

The data we have accumulated on the maximal cross-correlation between pairs of $W_1$ or $G_2$ arrays generated in fields of Germain prime size can be seen in Table \ref{t3}. Inspection of this table reveals that the $W_1$ column behaves approximately ``logarithmically'' with respect to the $p$ column, and this is reminiscent of the situation in Table \ref{t1}. It turns out that the same formula we suggested for Table \ref{t1}, namely $1+[\ln(p)]$, where $[\cdot]$ again denotes rounding, is an excellent fit here as well! More precisely, 13 out of the 19 entries in Table \ref{t3} (about 70\%) obey the formula without error, while all entries obey the formula with an error margin of $\pm 1$. Surprisingly then, it appears that the 2 seemingly unrelated experiments are somehow related after all, in the sense that the first order approximation of the results is the same.

Note that finding the maximal cross-correlation for pairs of $G_2$ arrays is a much heavier computational task than for pairs of $W_1$ arrays:

\begin{thm}
Let $C_G(p)$ and $C_W(p)$ denote the complexities involved in determining $\ds \max_{(u,v)}\max_{\left(f_1\neq f_2\right)}\Psi_{f_1,f_2}(u,v)$ for $G_2$ and $W_1$ permutations in $\FF(p)$, $p$ prime, respectively. Then,
\[ \frac{C_G(p)}{C_W(p)}\sim 2\frac{\phi^2(p-1)}{p},\text{ as } p\rightarrow \infty\]
\end{thm}

\begin{proof} We will present a sketch of the proof and work with asymptotic estimations. In $F(p)$ we can generate $(p-1)\phi(p-1)$ $W_1$ arrays and $\phi^2(p-1)$ $G_2$ arrays. In the case of $G_2$ arrays we need to compute $\ds \frac{\phi^2(p-1)(\phi^2(p-1)-1)}{2}$ cross-correlation arrays with $\ds (2(p-2)-1)^2\sim 4p^2$ entries in each. In the case of $W_1$ arrays, however, we only need to compute the vector of the cross-correlation array corresponding to $u=0$, which contains $2(p-1)-1\sim 2p$ entries (we obviously need to compute $\ds \frac{(p-1)\phi(p-1)((p-1)\phi(p-1)-1)}{2}$ such vectors).

The reason is not difficult to see: assume that we have to compute $\ds \Psi_{f_1,f_2}(u,v)$ with $u>0$ without loss of generality; clearly $\Psi_{f_1,f_2}(u,v)\leq \Psi_{f_1^u,f_2}(0,v)$, where $f_1^u$ denotes the permutation $f_1$ but with its entries shifted $u$ times to the left!

Overall, then:
\[\frac{C_G(p)}{C_W(p)}\sim \frac{\frac{1}{2}\phi^4(p-1)4p^2}{\frac{1}{2}p^2\phi^2(p-1)2p}=2\frac{\phi^2(p-1)}{p}\]
\end{proof}

This result explains why we are unable to offer as many results for $G_2$ arrays as for $W_1$ arrays in Table \ref{t3}. Note that, according to Theorem \ref{mainthmw}, whenever $p$ is a Germain prime and $f_1$, $f_2$ are $W_1$ permutations generated in $\FF(p)$, the maximal value $\ds \Psi_{f_1,f_2}(u,0)$ can attain is 2. Since Table \ref{t3} yields higher values for the maximal cross-correlation, necessarily this maximum is attained for some $v\neq 0$.

\begin{table}
\centering
\begin{tabular}{cc}

\begin{tabular}{|r|r|r|} \hline
$p$ & $W_1$ & $G_2$\\\hline
5 & 2& 2\\\hline
7 & 2& 2\\\hline
11 & 3& 4\\\hline
23 & 4& 6\\\hline
47 & 5& 8\\\hline
59 & 5& 12\\\hline
83 & 5& 9\\\hline
107 & 5& 10\\\hline
167 & 6& 12\\\hline
179 & 6& 12\\\hline
227 & 6& 13\\\hline
\end{tabular}
&
\begin{tabular}{c}
\begin{tabular}{|r|r|} \hline
$p$ & $W_1$\\\hline
263 & 7\\\hline
347 & 6\\\hline
359 & 6\\\hline
383 & 7\\\hline
467 & 7\\\hline
479 & 7\\\hline
503 & 7\\\hline
563 & 8\\\hline
\end{tabular}
\\
\begin{tabular}{c}
\\
\\
\\
\end{tabular}
\end{tabular}

\end{tabular}

\caption{\label{t3} The maximal cross-correlation between pairs of $W_1$ or $G_2$ arrays built in the first few fields of Germain prime size.}
\end{table}

\section{Summary and future work}

In this work we have presented the results of some of our numerical experiments on Costas arrays that we have hitherto been unable to account for, or even formulate relevant conjectures on; in that sense, the entire paper is a plan for future work. We chose the 3 most complex and intriguing experiments we have encountered so far, and presented all of the evidence we have gathered. It is our firm belief that these results are instances of as yet unexplored number theoretic or algebraic properties of (some families of) finite fields, so that further study of these matters will greatly benefit both pure mathematics and applications. We can only hope that we will successfully arouse the interest of a reader, perhaps better versed in the relevant techniques than ourselves, who will unravel the mysteries of these experiments. In such a contingency we ask that we be informed.

\section{Acknowledgements}

The author would like to thank Ken Taylor who produced most of Table \ref{t3} as part of his undergraduate final project. He would also like to thank Prof. Rod Gow for his coordinated attempts with the author to explain all 3 experiments presented here. Finally, he would like to thank Prof. Paul Curran, Dr. Scott Rickard, and John Healy for the long and useful discussions on these experiments.

\end{document}